\documentclass[10pt]{amsart}
\usepackage{amssymb,MnSymbol}
\usepackage{amsthm,amsmath}

\title[On the generalized associativity equation]{On the generalized associativity equation}\thanks{Corresponding author: Jean-Luc Marichal is with the Mathematics Research Unit, University of Luxembourg, 6, rue Coudenhove-Kalergi, L-1359 Luxembourg, Luxembourg. Email: jean-luc.marichal[at]uni.lu}

\author{Jean-Luc Marichal}
\address{Mathematics Research Unit, FSTC, University of Luxembourg, 6, rue Coudenhove-Kalergi, L-1359 Luxembourg, Luxembourg}
\email{jean-luc.marichal[at]uni.lu}

\author{Bruno Teheux}
\address{Mathematics Research Unit, FSTC, University of Luxembourg, 6, rue Coudenhove-Kalergi, L-1359 Luxembourg, Luxembourg}
\email{bruno.teheux[at]uni.lu}

\date{July 28, 2016}

\begin{document}

\theoremstyle{plain}
\newtheorem{theorem}{Theorem}%[section]% Supprimer [section] pour une numérotation linéaire
\newtheorem{lemma}[theorem]{Lemma}
\newtheorem{proposition}[theorem]{Proposition}
\newtheorem{corollary}[theorem]{Corollary}
\newtheorem{fact}[theorem]{Fact}
\newtheorem{problem}{Problem}
\newtheorem*{main}{Main Theorem}

\theoremstyle{definition}
\newtheorem{definition}[theorem]{Definition}
\newtheorem{example}[theorem]{Example}

\theoremstyle{remark}
\newtheorem*{conjecture}{onjecture}
\newtheorem{remark}{Remark}
\newtheorem{claim}{Claim}

\renewcommand{\S}{\mathcal{S}}

\newcommand{\N}{\mathbb{N}}                     % pos. integers
\newcommand{\R}{\mathbb{R}}                     % reals
\newcommand{\bfx}{\mathbf{x}}
\newcommand{\bfy}{\mathbf{y}}
\newcommand{\bfz}{\mathbf{z}}
\newcommand{\dom}{\mathrm{dom}}
\newcommand{\ran}{\mathrm{ran}}
\newcommand{\id}{\mathrm{id}}

\begin{abstract}
The so-called generalized associativity functional equation
$$
G(J(x,y),z) ~=~ H(x,K(y,z))
$$
has been investigated under various assumptions, for instance when the unknown functions $G$, $H$, $J$, and $K$ are real, continuous, and strictly monotonic in each variable. In this note we investigate the following related problem: given the functions $J$ and $K$, find every function $F$ that can be written in the form
$$
F(x,y,z) ~=~ G(J(x,y),z) ~=~ H(x,K(y,z))
$$
for some functions $G$ and $H$. We show how this problem can be solved when any of the inner functions $J$ and $K$ has the same range as one of its sections.
\end{abstract}



\maketitle

%---------------------------------------------------------------------------------------------- Section 1
\section{Introduction}

Let $X$, $Y$, $Z$, $U_J$, $U_K$, and $U$ be nonempty sets and consider the functional equation
\begin{equation}\label{eq:GAE-286we}
G(J(x,y),z) ~=~ H(x,K(y,z)),\qquad x\in X,~y\in Y,~z\in Z,
\end{equation}
where $J\colon X\times Y\to U_J$, $K\colon Y\times Z\to U_K$, $G\colon U_J\times Z\to U$, and $H\colon X\times U_K\to U$ are unknown functions. This functional equation, called the \emph{generalized associativity equation}, has been investigated under various solvability conditions, in particular when the unknown functions are real, continuous, and strictly monotonic in each variable (see, e.g, \cite{Acz66,Mak05} and the references therein).

In this paper we are interested in the following problem, which is closely related to that of solving the generalized associativity equation \eqref{eq:GAE-286we}. Throughout this paper we denote the domain and range of any function $f$ by $\dom(f)$ and $\ran(f)$, respectively.

\begin{problem}\label{prob:1}
Given two functions $J\colon X\times Y\to U_J$ and $K\colon Y\times Z\to U_K$, determine the class $\mathcal{F}_{J,K}$ of functions $F\colon X\times Y\times Z\to\ran(F)$ for which there exist $G\colon U_J\times Z\to\ran(F)$ and $H\colon X\times U_K\to\ran(F)$ such that
\begin{equation}\label{eq:GAE-286we2}
F(x,y,z) ~=~ G(J(x,y),z) ~=~ H(x,K(y,z)),\qquad x\in X,~y\in Y,~z\in Z.
\end{equation}
\end{problem}

Contrary to the problem of solving the generalized associativity equation, here we assume that the inner functions $J$ and $K$ are given beforehand and we search for all functions $F$ which have the form given in \eqref{eq:GAE-286we2}. For instance, searching for the real functions $F\colon\R^3\to\R$ that can be expressed in the form
$$
F(x,y,z) ~=~ G(x-y,z) ~=~ H(x,y-z),\qquad x,y,z\in\R,
$$
for some functions $G,H\colon\R^2\to\R$ is a special case of Problem~\ref{prob:1}. As we will see in Example~\ref{ex:xz-y}, these functions are all of the form
$$
F(x,y,z) ~=~ f(x-y+z),
$$
where $f\colon\R\to\R$ is an arbitrary function.

The following problem gives an equivalent but simpler reformulation of Problem~\ref{prob:1}, where the functions $G$ and $H$ are not explicitly involved.

\begin{problem}\label{prob:2}
Given two functions $J\colon X\times Y\to U_J$ and $K\colon Y\times Z\to U_K$, determine the class $\mathcal{F}_{J,K}$ of functions $F\colon X\times Y\times Z\to\ran(F)$ satisfying the conditions
\begin{eqnarray*}
J(x,y) ~=~ J(x',y') & ~~\Rightarrow ~~ & F(x,y,z) ~=~ F(x',y',z)\\
K(y,z) ~=~ K(y',z') & ~~\Rightarrow ~~ & F(x,y,z) ~=~ F(x,y',z')
\end{eqnarray*}
for all $x,x'\in X$, all $y,y'\in Y$, and $z,z'\in Z$.
\end{problem}

It is easy to see that Problem~\ref{prob:1} and Problem~\ref{prob:2} are equivalent in the sense that they define the same class $\mathcal{F}_{J,K}$ of functions. We also observe that $\mathcal{F}_{J,K}$ is never empty since it contains all the constant functions. More generally, we have the following fact.

\begin{fact}\label{fact:fF}
If $F\in\mathcal{F}_{J,K}$, then $f\circ F\in\mathcal{F}_{J,K}$ for every function $f$ defined on $\ran(F)$.
\end{fact}

%We also consider the following alternative problem, which is a variant of Problem~\ref{prob:1}.
%
%\begin{problem}\label{prob:3}
%Given functions $J\colon X\times Y\to U_J$, $K\colon Y\times Z\to U_K$, $\alpha\colon U_J\to X\times Y$, and $\beta\colon U_K\to Y\times Z$, determine the class $\mathcal{F}_{J,K}^{\alpha,\beta}$ of functions $F\colon X\times Y\times Z\to U$ satisfying the equation
%\begin{equation*}%\label{eq:GAE-286we23}
%F(x,y,z) ~=~ F(\alpha\circ J(x,y),z) ~=~ F(x,\beta\circ K(y,z)),\quad x\in X,~y\in Y,~z\in Z.
%\end{equation*}
%\end{problem}
%
%It is very easy to see that Problem~\ref{prob:3} is a special case of Problem~\ref{prob:1}. That is, the inclusion $\mathcal{F}_{J,K}^{\alpha,\beta}\subseteq\mathcal{F}_{J,K}$ holds for any functions $J$, $K$, $\alpha$, and $\beta$.

Solving Problem~\ref{prob:1}, or equivalently Problem~\ref{prob:2}, seems not easy in general. However, solutions can be found as soon as certain assumptions are made on the functions $J$ and $K$. In Section 2 we show how this problem can be solved whenever any of the functions $J$ and $K$ has the same range as one of its sections.
%Considering the function $K$ for instance, this condition simply means that there exists $a\in Z$ such that for every $(y,z)\in Y\times Z$ there exists $y'\in Y$ such that $K(y,z)=K(y',a)$.
In Section 3 we focus on the special case where $X=Z=A$ and $Y=A^{n-2}$ for some nonempty set $A$ and some integer $n\geqslant 3$ (in which case any function in $\mathcal{F}_{J,K}$ is defined on the Cartesian power $A^n$) and we provide conditions on $J$ and $K$ for the functions in $\mathcal{F}_{J,K}$ to be expressible in terms of their diagonal sections (i.e., every $F\in\mathcal{F}_{J,K}$ is of the form $F=\delta_F\circ M$ for some function $M\colon A^n\to A$, where $\delta_F\colon A\to A$ is defined by $\delta_F(x)=F(x,\ldots,x)$).

We use the following notation. The identity function on any nonempty set $E$ is denoted by $\id_E$. We denote the set of quasi-inverses of a function $f$ by $Q(f)$, where a quasi-inverse $g$ of a function $f$ is defined by the conditions (see, e.g., \cite[Sect.~2.1]{SchSkl83})
$$
f\circ g|_{\ran(f)}=\id_{\ran(f)}\qquad\mbox{and}\qquad\ran(g|_{\ran(f)})=\ran(g).
$$
Throughout this paper we assume that every function has at least one quasi-inverse. It is well known that this assumption is equivalent to the Axiom of Choice. Recall also that the relation of being quasi-inverse is symmetric: if $g \in Q(f)$ then $f \in Q(g)$; moreover, we have $\ran(g)\subseteq\dom(f)$ and $\ran(f)\subseteq\dom(g)$ and the functions $f|_{\ran(g)}$ and $g|_{\ran(f)}$ are one-to-one (in particular if $\ran(g)=\dom(f)$ and $\ran(f)=\dom(g)$, then $f$ and $g$ are inverses of each other).

%By definition, if $g\in Q(f)$ then $g|_{\ran(f)}\in Q(f)$. Thus we can always restrict the domain of any quasi-inverse $g\in Q(f)$ to $\ran(f)$.
%These ``restricted'' quasi-inverses, also called \emph{right-inverses}, are then simply characterized by the condition $f\circ g|_{\ran(f)}=\id_{\ran(f)}$, which can be rewritten as
%\begin{equation*}%\label{eq:defqi11}
%g(y)\in f^{-1}\{y\},\qquad \forall y\in\ran(f).
%\end{equation*}

\begin{fact}\label{fact:11}
If $g\in Q(f)$ and $\ran(h)\subseteq\ran(f)$, then $f\circ g\circ h=h$.
\end{fact}

\begin{remark}
Consider the class $\mathcal{F}_{J,K}$ as defined in Problem~\ref{prob:1} and let $F\in\mathcal{F}_{J,K}$. Then we have $G(a,b) = F(\phi(a),b)$ for every $\phi\in Q(J)$ and every $(a,b)\in\ran(J)\times Z$. Therefore, $G$ is completely determined from $F$. Similarly, we have $H(a,b)=F(a,\psi(b))$ for every $\psi\in Q(K)$ and every $(a,b)\in X\times\ran(K)$, and hence $H$ is completely determined from $F$. Thus, when such quasi-inverses $\phi$ and $\psi$ can be given explicitly, Problem~\ref{prob:1} amounts to solving the generalized associativity equation \eqref{eq:GAE-286we} for given functions $J$ and $K$ and an arbitrary set $U$.
\end{remark}

%--------------------------------------------------------------------------------------------- Section
\section{Main results}

For any $a\in Z$ we define the section $K_2^a\colon Y\to U_K$ of $K$ as the function $K_2^a(y)=K(y,a)$. The following theorem provides a first step in the resolution of Problem~\ref{prob:1} whenever $\ran(K)=\ran(K_2^a)$ for some $a\in Z$.

\begin{theorem}\label{thm:main1}
Assume that $\ran(K)=\ran(K_2^a)$ for some $a\in Z$ and let $F\in \mathcal{F}_{J,K}$. Then there exists $f\colon U_J\to\ran(F)$ such that $F=f\circ R_k$ for every $k\in Q(K_2^a)$, where $R_k\colon X\times Y\times Z\to U_J$ is defined by
$$
R_k(x,y,z) ~=~ J(x,k\circ K(y,z)).
$$
\end{theorem}

\begin{proof}
Let $F\in \mathcal{F}_{J,K}$. Then, there exist $G\colon U_J\times Z\to\ran(F)$ and $H\colon X\times U_K\to\ran(F)$ such that \eqref{eq:GAE-286we2} holds. Let $f=G_2^a$ and $k\in Q(K_2^a)$. For any $(y,z)\in Y\times Z$, by Fact~\ref{fact:11} we have
$$
K(y,z) ~=~ K_2^a\circ k\circ K(y,z) ~=~ K(k\circ K(y,z),a).
$$
For every $(x,y,z)\in X\times Y\times Z$, we then have
\begin{eqnarray*}
F(x,y,z) &=& H(x,K(y,z)) ~=~ H(x,K(k\circ K(y,z),a))\\
&=& G(J(x,k\circ K(y,z)),a) ~=~ f\circ R_k(x,y,z),
\end{eqnarray*}
which completes the proof.
\end{proof}

\begin{remark}
We observe that, although the quasi-inverse $k$ of $K_2^a$ need not be unique, the identity $F=f\circ R_k$ in Theorem~\ref{thm:main1} does not depend on the choice of this quasi-inverse.
\end{remark}

\begin{example}\label{ex:xz-y}
Let $F\colon\R^3\to\R$ be a function for which there exist $G,H\colon\R^2\to\R$ such that
\begin{equation}\label{eq:Ex4}
F(x,y,z) ~=~ G(x-y,z) ~=~ H(x,y-z),\qquad x,y,z\in\R.
\end{equation}
Searching for all possible such functions $F$ reduces to describing the functions in $\mathcal{F}_{J,K}$ that range in $\R$, where $J$ and $K$ are defined by $J(x,y)=x-y$ and $K(y,z)=y-z$. Since $K_2^0=\id_{\R}$, we have $\ran(K_2^0)=\R=\ran(K)$ and hence we can apply Theorem~\ref{thm:main1} with $a=0$. We then have $k=(K_2^0)^{-1}=\id_{\R}$ and $R_k(x,y,z)=x-y+z$. Therefore any function $F\in\mathcal{F}_{J,K}$ ranging in $\R$ is of the form
\begin{equation}\label{eq:Ex4a}
F(x,y,z)~=~f(x-y+z),\qquad x,y,z\in\R,
\end{equation}
for some $f\colon\R\to\R$. Conversely any such function clearly lies in $\mathcal{F}_{J,K}$. Therefore we necessarily have
$$
\{F\in\mathcal{F}_{J,K}\mid\ran(F)\subseteq\R\} ~=~ \{(x,y,z)\mapsto f(x-y+z)\mid f\colon\R\to\R\}.
$$
Finally, setting $y=0$ in \eqref{eq:Ex4} and \eqref{eq:Ex4a} we obtain $G(x,z)=H(x,-z)=f(x+z)$ for every $x,z\in\R$.\qed
\end{example}

\begin{example}\label{ex:join}
Assume that $(A,\vee)$ is a bounded join-semilattice, with $0$ as the least element and let $F\colon A^n\to\ran(F)$ be a function for which there exist $G,H\colon A^2\to\ran(F)$ such that
$$
F(x_1,\ldots,x_n) ~=~ G(x_1\vee\cdots\vee x_{n-1},x_n) ~=~ H(x_1,x_2\vee\cdots\vee x_n).
$$
The class of all possible functions $F$ satisfying this condition is nothing other than the set $\mathcal{F}_{J,K}$, where the functions $J,K\colon A^{n-1}\to A$ are defined by $J(x_1,\ldots,x_{n-1})=K(x_1,\ldots,x_{n-1})=x_1\vee\cdots\vee x_{n-1}$. Using Theorem~\ref{thm:main1} with $a=0$, we can easily see that
$$
\mathcal{F}_{J,K} ~=~ \{(x_1,\ldots,x_n)\mapsto f(x_1\vee\cdots\vee x_n)\mid \dom(f)=A\}.\qed
$$
\end{example}

The following theorem is the dual version of Theorem~\ref{thm:main1}. The proof is similar to that of Theorem~\ref{thm:main1} and hence is omitted. For any $b\in X$ we define the section $J_1^b\colon Y\to U_J$ of $J$ as the function $J_1^b(y)=J(b,y)$.

\begin{theorem}\label{thm:main2}
Assume that $\ran(J)=\ran(J_1^b)$ for some $b\in X$ and let $F\in \mathcal{F}_{J,K}$. Then there exists $g\colon U_K\to\ran(F)$ such that $F=g\circ S_j$ for every $j\in Q(J_1^b)$, where $S_j\colon X\times Y\times Z\to U_K$ is defined by
$$
S_j(x,y,z) ~=~ K(j\circ J(x,y),z).
$$
\end{theorem}

We observe that each of Theorems~\ref{thm:main1} and \ref{thm:main2} provides only necessary conditions for a function to be in $\mathcal{F}_{J,K}$. Examples~\ref{ex:xz-y} and \ref{ex:join} show that the use of only one of these theorems may sometimes be sufficient to derive a complete characterization of the class $\mathcal{F}_{J,K}$. The following example shows that using both theorems may somewhat simplify the quest for such a characterization.

\begin{example}\label{ex:prodsum}
Let $F\colon\R^3\to\ran(F)$ be a function for which there exist $G,H\colon\R^2\to\ran(F)$ such that
$$
F(x,y,z) ~=~ G(xy,z) ~=~ H(x,y+z),\qquad x,y,z\in\R.
$$
Using both Theorems~\ref{thm:main1} and \ref{thm:main2} with $J(x,y)=xy$, $K(y,z)=y+z$, $a=0$, and $b=1$, we obtain the expressions
$$
F(x,y,z) ~=~ f(xy+xz) ~=~ g(xy+z)
$$
for some functions $f,g\colon\R\to\ran(F)$. Setting $y=0$ and $z=1$ in these equations shows that $f=g$ must be a constant function. Therefore $\mathcal{F}_{J,K}$ consists of the class of constant functions. Note that using Theorem~\ref{thm:main1} only would have been sufficient here. Indeed, taking $a=0$ and then $a=1$ would lead to the identity $f(xy+xz)=f'(xy+xz-x)$ for some functions $f,f'\colon\R\to\ran(F)$, from which we would reach the same conclusion by setting $y=1$ and $z=0$.\qed
\end{example}

Example~\ref{ex:prodsum} may suggest that the set $\mathcal{F}_{J,K}$ reduces to the class of constant functions whenever the functions $J$ and $K$ do not coincide. To see that this is not true, just replace $J$ and $K$ in Example~\ref{ex:prodsum} with the functions $J(x,y)=y$ and $K(y,z)=y$, respectively. Any $F\in\mathcal{F}_{J,K}$ is then of the form $F(x,y,z)=f(y)$ for some function $f\colon\R\to\ran(F)$.

Interestingly, $\mathcal{F}_{J,K}$ may reduce to the class of constant functions even if $J$ and $K$ coincide. The following example illustrates this fact.

\begin{example}\label{ex:max1}
Let $A=\left[0,+\infty\right[$ and let $F\colon A^3\to\ran(F)$ be a function for which there exist $G,H\colon A^2\to\ran(F)$ such that
$$
F(x,y,z) ~=~ G(\max(1,x+y),z) ~=~ H(x,\max(1,y+z)),\qquad x,y,z\in A.
$$
Using both Theorems~\ref{thm:main1} and \ref{thm:main2} with $J(x,y)=K(x,y)=\max(1,x+y)$ and $a=b=0$ and choosing $j=k=\id_{\left[1,+\infty\right[}$, we obtain the expressions
$$
F(x,y,z) ~=~ f(\max(1+x,x+y+z)) ~=~ g(\max(1+z,x+y+z))
$$
for some functions $f,g\colon\left[1,+\infty\right[\to\ran(F)$. Setting first $x\in [0,1]$ and $y=z=0$ and then $x=y=0$ and $z\in [0,1]$ in these identities, we obtain that $f=g$ is constant on $[1,2]$. Then, setting $x\geqslant 1$ and $y=z=0$ and then $x=y=0$ and $z\geqslant 1$, we obtain that $f=g$ is constant on $\left[1,+\infty\right[$. Therefore $\mathcal{F}_{J,K}$ consists of the class of constant functions.\qed
\end{example}

The following two propositions give sufficient conditions on the functions $R_k$ and $S_j$ (as defined in Theorems~\ref{thm:main1} and \ref{thm:main2}) to obtain a characterization of the class $\mathcal{F}_{J,K}$.

\begin{proposition}\label{prop:CSRk}
Assume that $\ran(K)=\ran(K_2^a)$ for some $a\in Z$. Let $R_k\colon X\times Y\times Z\to U_J$ be defined as in Theorem~\ref{thm:main1}. If $R_k\in\mathcal{F}_{J,K}$, then $\mathcal{F}_{J,K}=\{f\circ R_k\mid \dom(f)=U_J\}$.
\end{proposition}

\begin{proof}
Inclusion `$\subseteq$' follows from Theorem~\ref{thm:main1}. Inclusion `$\supseteq$' follows from both the hypothesis and Fact~\ref{fact:fF}.
\end{proof}

\begin{proposition}\label{prop:CSSj}
Assume that $\ran(J)=\ran(J_1^b)$ for some $b\in X$. Let $S_j\colon X\times Y\times Z\to U_K$ be defined as in Theorem~\ref{thm:main2}. If $S_j\in\mathcal{F}_{J,K}$, then $\mathcal{F}_{J,K}=\{g\circ S_j\mid \dom(g)=U_K\}$.
\end{proposition}

\begin{remark}
Finding necessary and sufficient conditions on functions $J$ and $K$ for $R_k$ (or $S_j$) to be in $\mathcal{F}_{J,K}$ remains an interesting problem.
\end{remark}

The following proposition states that if the functions $f\circ R_k$ and $g\circ S_j$ defined in Theorems~\ref{thm:main1} and \ref{thm:main2} are equal, then they belong to the class $\mathcal{F}_{J,K}$.

\begin{proposition}\label{prop:sd76f5}
Assume that $\ran(K)=\ran(K_2^a)$ and $\ran(J)=\ran(J_1^b)$ for some $a\in Z$ and $b\in X$. Let $R_k\colon X\times Y\times Z\to U_J$ and $S_j\colon X\times Y\times Z\to U_K$ be defined as in Theorems~\ref{thm:main1} and \ref{thm:main2}. If $f\circ R_k=g\circ S_j$ for some functions $f$ and $g$ such that $\dom(f)=U_J$ and $\dom(g)=U_K$, then $f\circ R_k\in\mathcal{F}_{J,K}$.
\end{proposition}

%\begin{proof}
%Let $x,x'\in X$, $y,y'\in Y$, and $z,z'\in Z$. If $J(x,y)=J(x',y')$, then we have
%\begin{eqnarray*}
%g\circ S_j(x,y,z) &=& g\circ K(j\circ J(x,y),z) ~=~ g\circ K(j\circ J(x',y'),z)\\
%&=& g\circ S_j(x',y',z).
%\end{eqnarray*}
%Dually, if $K(y,z)=K(y',z')$, then $f\circ R_k(x,y,z)=f\circ R_k(x,y',z')$. Therefore, the function $f\circ R_k=g\circ S_j$ lies in $\mathcal{F}_{J,K}$ (cf.\ Problem~\ref{prob:2}).
%\end{proof}

\begin{proof}
Since the identity $f\circ R_k=g\circ S_j$ can be rewritten as condition \eqref{eq:GAE-286we} for some function $G$ defined on $U_J\times Z$ and some function $H$ defined on $X\times U_K$, the function $f\circ R_k$ is necessarily in $\mathcal{F}_{J,K}$.
\end{proof}

\begin{remark}\label{rem:f8s6d}
Proposition~\ref{prop:sd76f5} is particularly useful, when combined with any of the Propositions~\ref{prop:CSRk} and \ref{prop:CSSj}, if for instance $U_J=U_K$ and $f=g$ is the identity function (or a one-to-one function by Fact~\ref{fact:fF}). Example~\ref{ex:11} illustrates this observation.
\end{remark}

\begin{example}\label{ex:11}
Consider the class $\mathcal{F}_{J,K}$, where $J,K\colon [0,1]^2\to [0,1]$ are defined by $J(x,y)=K(x,y)=\frac{1}{2}\max(1,x+y)$. Consider also the functions $R_k$ and $S_j$ defined in Theorems~\ref{thm:main1} and \ref{thm:main2} by choosing the values $a=b=1$ and the functions $j,k\colon [\frac{1}{2},1]\to [0,1]$ defined by $j(x)=k(x)=2x-1$. Then we have
$$
R_k(x,y,z) ~=~ S_j(x,y,z) ~=~ \textstyle{\frac{1}{2}\max(1,x+y+z-1)}.
$$
By Proposition~\ref{prop:sd76f5} (and in view of Remark~\ref{rem:f8s6d}), we can immediately see that
$$
\mathcal{F}_{J,K} ~=~ \{(x,y,z)\mapsto \textstyle{f(\frac{1}{2}\max(1,x+y+z-1))}\mid \dom(f)=\textstyle{[\frac{1}{2},1]}\}.\qed
$$
%or equivalently,
%$$
%\mathcal{F}_{J,K} ~=~ \{(x,y,z)\mapsto f(\max(2,x+y+z))\mid f\colon [2,3]\to\R\}.\qed
%$$
\end{example}

We observe that Problem~\ref{prob:1} can also be generalized to functions $J$ and $K$ that are defined on subsets of $X\times Y$ and $Y\times Z$, respectively. Such a generalization can be useful for instance when the assumption of Theorem~\ref{thm:main1} is not satisfied (i.e., when $\ran(K)\neq\ran(K_2^a)$ for all $a\in Z$). For the interested reader we elaborate on this generalization in the Appendix.

%Let us now turn to the resolution of Problem~\ref{prob:3}. Proceeding exactly as in Theorems~\ref{thm:main1} and \ref{thm:main2}, we immediately obtain the following two results.
%
%\begin{theorem}\label{thm:main3}
%Assume that $\ran(K)=\ran(K_2^a)$ for some $a\in Z$ and let $F\in \mathcal{F}_{J,K}^{\alpha,\beta}$. Then there exists $f\colon X\times Y\to U$ such that $F=f\circ\alpha\circ R_k$ for every $k\in Q(K_2^a)$, where $R_k\colon X\times Y\times Z\to U_J$ is defined as
%$$
%R_k(x,y,z) ~=~ J(x,k\circ K(y,z)).
%$$
%Moreover, we have $f=F_3^a$.
%\end{theorem}
%
%\begin{theorem}\label{thm:main4}
%Assume that $\ran(J)=\ran(J_1^b)$ for some $b\in X$ and let $F\in \mathcal{F}_{J,K}^{\alpha,\beta}$. Then there exists $g\colon Y\times Z\to U$ such that $F=g\circ\beta\circ S_j$ for every $j\in Q(J_1^b)$, where $S_j\colon X\times Y\times Z\to U_K$ is defined as
%$$
%S_j(x,y,z) ~=~ K(j\circ J(x,y),z).
%$$
%Moreover, we have $g=F_1^b$.
%\end{theorem}

%---------------------------------------------------------------------------------------------- Section
\section{When the domain of $F$ is a Cartesian power}

We now particularize Problem~\ref{prob:1} to the case where $X=Z=A$ and $Y=A^{n-2}$ for some nonempty set $A$ and some integer $n\geqslant 3$. We then have $X\times Y\times Z=A^n$ and both functions $J$ and $K$ have $n-1$ arguments (like in Example~\ref{ex:join}).

Recall that the \emph{diagonal section} of a function $F\colon A^n\to\ran(F)$ is the function $\delta_F\colon A\to\ran(F)$ defined by $\delta_F(x)=F(x,\ldots,x)$. Also, a function $F\colon A^n\to A$ is said to be \emph{range-idempotent} if $\delta_F\circ F=F$. It is said to be \emph{idempotent} if $\delta_F=\id_A$.

In this section we provide conditions on $J$ and $K$ for each function $F$ in $\mathcal{F}_{J,K}$ to be expressible as $F=\delta_F\circ M$ for some function $M\colon A^n\to A$. Under idempotence and nondecreasing monotonicity (assuming $A$ is an ordered set), such a function $M$ is then called a \emph{Chisini mean associated with $F$} (see \cite{Mar10}). This observation could be useful in applications where aggregation functions, and especially mean functions, are considered.

Let us first consider an important but simple lemma.

\begin{lemma}\label{lemma:s68fd}
Let $R\colon A^n\to U$ be a function such that $\ran(R)=\ran(\delta_R)$ and consider the functions $f\colon U\to V$ and $F=f\circ R$. The following assertions hold.
\begin{enumerate}
\item[(a)] We have $F=\delta_F\circ r\circ R$ for every $r\in Q(\delta_R)$.

\item[(b)] If $F'=f'\circ R$ satisfies $\delta_F=\delta_{F'}$ for some $f'\colon U\to V$, then we have $F=F'$.

\item[(c)] If $F$ is idempotent, then $F=r\circ R$ for every $r\in Q(\delta_R)$.

\item[(d)] If $r\circ R$ is not idempotent for some $r\in Q(\delta_R)$, then $F$ is not idempotent.

\item[(e)] For every $r\in Q(\delta_R)$, the function $r\circ R$ is range-idempotent (i.e., $r\circ\delta_R\circ r\circ R=r\circ R$). It is idempotent if and only if $\delta_R$ is one-to-one.
\end{enumerate}
\end{lemma}

\begin{proof}
By Fact~\ref{fact:11} we have $\delta_R\circ r\circ R=R$, which proves assertion (e). We also derive the identities $F = f\circ R = f\circ \delta_R\circ r\circ R = \delta_F\circ r\circ R$, which prove assertion (a). Assertions (b) and (c) immediately follow from (a). Assertion (d) follows from (c).
\end{proof}

Whenever its assumptions are satisfied, Lemma~\ref{lemma:s68fd} provides interesting properties of function $F$. Assertions (a) and (c) give an explicit expression of $F$ in terms of its diagonal section. Assertion (b) shows that $F$ depends only on $\delta_F$ and $R$. Assertion (d) is nothing other than the contrapositive of assertion (c). Finally, assertion (e) reveals a surprising property of $r\circ R$.

\begin{example}\label{ex:SAM}
Let $R\colon\R^n\to\R$ be the sum function $R(x_1,\ldots,x_n)=\sum_{i=1}^nx_i$. Assertion (a) of Lemma~\ref{lemma:s68fd} shows that, for any function $f\colon\R\to\R$, the function $F=f\circ R$ can be written as
$$
F(x_1,\ldots,x_n) ~=~ \delta_F\Big(\frac{1}{n}\,\sum_{i=1}^nx_i\Big),\qquad x_1,\ldots,x_n\in\R.\qed
$$
\end{example}

For any function $F\colon A^n\to U$, any $k\in\{1,\ldots,n\}$, and any $a\in A$, we define the section $F_k^a\colon A^{n-1}\to U$ of $F$ by
$$
F_k^a(x_1,\ldots,x_{k-1},x_{k+1},\ldots,x_n) ~=~ F(x_1,\ldots,x_{k-1},a,x_{k+1},\ldots,x_n).
$$
Combining Theorem~\ref{thm:main1} with Lemma~\ref{lemma:s68fd}, we obtain the following result. First define $\mathcal{F}^A_{J,K}=\{F\in\mathcal{F}_{J,K}\mid \ran(F)\subseteq A\}$.

\begin{theorem}\label{thm:14}
Assume $\ran(K)=\ran(K_{n-1}^a)$ for some $a\in A$ and let $F\in\mathcal{F}_{J,K}$. Assume also that $\ran(R)=\ran(\delta_R)$, where $R$ is one of the functions $R_k$ ($k\in Q(K_{n-1}^a)$) defined in Theorem~\ref{thm:main1}. Then the assertions (a)--(e) of Lemma~\ref{lemma:s68fd} hold (where $U$ and $V$ are to be replaced with $U_J$ and $\ran(F)$, respectively). If $F\in\mathcal{F}^A_{J,K}$, then for every $r\in Q(\delta_R)$ we have $r\circ R\in\mathcal{F}^A_{J,K}$ if and only if
$$
\mathcal{F}^A_{J,K} ~=~ \{f\circ r\circ R\mid f\colon A\to A\}.
$$
\end{theorem}

\begin{proof}
By Theorem~\ref{thm:main1} there exists $f\colon U_J\to\ran(F)$ such that $F=f\circ R$ and hence Lemma~\ref{lemma:s68fd} applies. Let us now establish the last part of the theorem. To see that the condition is sufficient, just take $f=\id_A$. Let us show that it is necessary. The inclusion `$\supseteq$' follows from Fact~\ref{fact:fF}. To see that the inclusion `$\subseteq$' holds, let $F'\in\mathcal{F}^A_{J,K}$. Then $F'=\delta_{F'}\circ r\circ R$ (by assertion (a)) and hence we can take $f=\delta_{F'}$.
\end{proof}

The dual statement of Theorem~\ref{thm:14} can be derived immediately. We then have the following theorem.

\begin{theorem}\label{thm:15}
Assume that $\ran(J)=\ran(J_1^b)$ for some $b\in A$ and let $F\in\mathcal{F}_{J,K}$. Assume also that $\ran(S)=\ran(\delta_S)$, where $S$ is one of the functions $S_j$ ($j\in Q(J_1^b)$) defined in Theorem~\ref{thm:main2}. Then the assertions (a)--(e) of Lemma~\ref{lemma:s68fd} hold (where $R$, $r$, $U$, and $V$ are to be replaced with $S$, $s$, $U_K$ and $\ran(F)$, respectively). If $F\in\mathcal{F}^A_{J,K}$, then for every $s\in Q(\delta_S)$ we have $s\circ S\in\mathcal{F}^A_{J,K}$ if and only if
$$
\mathcal{F}^A_{J,K} ~=~ \{f\circ s\circ S\mid f\colon A\to A\}.
$$
\end{theorem}

\begin{example}\label{ex:17e}
Considering again Example~\ref{ex:11}, where
$$
R(x,y,x) ~=~ \textstyle{\frac{1}{2}{\,}\max(1,x+y+z-1)}
$$
on $[0,1]^3$, we clearly see that $\ran(R)=\ran(\delta_R)=[\frac{1}{2},1]$. Applying Theorem~\ref{thm:14}, with $r(x)=\frac{2x+1}{3}$ on $[\frac{1}{2},1]$ for instance, we obtain
$$
r\circ R(x,y,z) ~=~ \textstyle{\frac{1}{3}{\,}\max(2,x+y+z)}
$$
and for any $F\in\mathcal{F}_{J,K}$ we have $F(x,y,z) = \delta_F(\frac{1}{3}{\,}\max(2,x+y+z))$. By identifying the variables in the latter identity we then obtain $\delta_F(x)=\delta_F(\frac{1}{3}{\,}\max(2,3x))$, which shows that $\delta_F$ is constant on $[0,\frac{2}{3}]$.\qed
\end{example}

\begin{remark}
Let $A$ be a nonempty real interval possibly unbounded. Recall that if a function $F\colon A^n\to\R$ is nondecreasing in each variable and satisfies $\ran(F)=\ran(\delta_F)$, then there always exists a function $M\colon A^n\to A$ (called a \emph{Chisini mean}) that is idempotent and nondecreasing in each variable such that $F=\delta_F\circ M$ (see \cite{Mar10} for a constructive proof). For instance, considering again the functions in Example~\ref{ex:17e}, we can write $F(x,y,z)=\delta_{F}(\frac{x+y+z}{3})$, where $\delta_F$ is constant on $[0,\frac{2}{3}]$.
\end{remark}

We observe that the function $r\circ R$ is idempotent in Example~\ref{ex:SAM} while it is not in Example~\ref{ex:17e}. Actually, under the assumptions of Theorem~\ref{thm:14}, the function $r\circ R$ is idempotent whenever there exists $F\in\mathcal{F}_{J,K}$ such that $\delta_F$ is one-to-one. Indeed, we then have $r\circ R=\delta_F^{-1}\circ F$ and hence $\delta_{r\circ R}=\delta_F^{-1}\circ\delta_F=\id_A$, which shows that $r\circ R$ is idempotent. Clearly, the dual version of this fact can be derived by considering the assumptions of Theorem~\ref{thm:15}.

We also have the following result.

\begin{proposition}
Under the assumptions of both Theorems~\ref{thm:14} and \ref{thm:15}, the following assertions are equivalent.
\begin{enumerate}
\item[(i)] There exists $F\in\mathcal{F}_{J,K}$ such that $\delta_F$ is one-to-one.

\item[(ii)] $r\circ R=s\circ S$ is idempotent.

\item[(iii)] $r\circ R$ is idempotent and lies in $\mathcal{F}_{J,K}$.

\item[(iv)] $s\circ S$ is idempotent and lies in $\mathcal{F}_{J,K}$.
\end{enumerate}
\end{proposition}

\begin{proof}
Clearly (iii) or (iv) implies (i). Let us prove that (i) implies (ii). As observed above, we have $r\circ R=\delta_F^{-1}\circ F=s\circ S$ and hence the function $r\circ R=s\circ S$ is idempotent. Finally, (ii) implies (iii) and (iv) by Proposition~\ref{prop:sd76f5}.
\end{proof}

\begin{remark}
We observe that the proof of Lemma~\ref{lemma:s68fd} does not rely on the very concept of diagonal section. Actually, Lemma~\ref{lemma:s68fd} can be easily generalized as follows. Consider the functions $R\colon X\to U$, $f\colon U\to V$, $F=f\circ R$, and $\Pi\colon X\to X$ and assume that $\ran(R)=\ran(R\circ\Pi)$. Denote by $\ker(f)$ the \emph{kernel} of any function $f$, that is, the relation $\{(a,b)\in\dom(f)^2\mid f(a)=f(b)\}$. Then the following assertions hold.
\begin{enumerate}
\item[(a)] We have $F=F\circ\Pi\circ r\circ R$ for every $r\in Q(R\circ\Pi)$.

\item[(b)] If $F'=f'\circ R$ satisfies $F\circ\Pi=F'\circ\Pi$ for some $f'\colon U\to V$, then we have $F=F'$.

\item[(c)] If $F\circ\Pi=\Pi$, then $F=\Pi\circ r\circ R$ for every $r\in Q(R\circ\Pi)$.

\item[(d)] If $\Pi\circ r\circ R\circ\Pi\neq\Pi$ for some $r\in Q(R\circ\Pi)$, then we have $F\circ\Pi\neq\Pi$.

\item[(e)] If $\Pi\circ\Pi =\Pi$, then $T_r\circ\Pi\circ T_r=T_r$ for every $r\in Q(R\circ\Pi)$, where $T_r=\Pi\circ r\circ R$. In this case we have $T_r\circ\Pi=\Pi$ if and only if $\ker(T_r\circ\Pi)=\ker(\Pi)$.
\end{enumerate}
\end{remark}

\appendix
%---------------------------------------------------------------------------------------------- Appendix
\section*{Appendix}

We consider a generalization of Problem~\ref{prob:1} in which the functions $J$ and $K$ are defined on subsets of $X\times Y$ and $Y\times Z$, respectively.

\begin{problem}\label{prob:3}
Given two functions $J\colon D_J\to U_J$ and $K\colon D_K\to U_K$, where $D_J\subseteq X\times Y$ and $D_K\subseteq Y\times Z$, determine the class $\mathcal{F}_{J,K}$ of functions $F\colon D_{J,K}\to\ran(F)$, where $D_{J,K}=\{(x,y,z)\mid (x,y)\in D_J~\text{and}~(y,z)\in D_K\}$, for which there exist $G\colon U_J\times Z\to\ran(F)$ and $H\colon X\times U_K\to\ran(F)$ such that
$$%\begin{equation}\label{eq:GAE-286we2a}
F(x,y,z) ~=~ G(J(x,y),z) ~=~ H(x,K(y,z)),\qquad (x,y,z)\in D_{J,K}.
$$%\end{equation}
\end{problem}

This generalization of Problem~\ref{prob:1} can be useful for instance when the assumption of Theorem~\ref{thm:main1} is not satisfied (i.e., when $\ran(K)\neq\ran(K_2^a)$ for all $a\in Z$). Indeed, it then may be possible to restrict the domain of $K$ to a subset $D_K\subseteq Y\times Z$ on which the assumption is satisfied. These situations are illustrated in the following results (whose proofs can be obtained by a simple adaptation of Theorems~\ref{thm:main1} and \ref{thm:main2}) and examples.

If a function $f$ is defined on $D\subseteq X\times Y$ and if $a\in Y$, then we denote by $f_2^a$ the function defined on $\{x\in X\mid (x,a)\in D\}$ by $f_2^a(x)=f(x,a)$. Similarly, if $b\in X$, then we denote by $f_1^b$ the function defined on $\{y\in Y\mid (b,y)\in D\}$ by $f_1^b(y)=f(b,y)$.

\begin{theorem}\label{thm:main1a}
Under the notation of Problem~\ref{prob:3}, assume that $\ran(K)=\ran(K_2^a)$ for some $a\in Z$ and let $F\in\mathcal{F}_{J,K}$. Then there exists $f\colon U_J\to\ran(F)$ such that for every $k\in Q(K_2^a)$ we have $F(x,y,z)=f\circ R_k(x,y,z)$ for every $(x,y,z)\in D_{J,K}$ such that $(x,k\circ K(y,z))\in D_J$, where $R_k\colon X\times Y\times Z\to U_J$ is defined by $R_k(x,y,z) = J(x,k\circ K(y,z))$.
\end{theorem}

\begin{theorem}\label{thm:main2a}
Under the notation of Problem~\ref{prob:3}, assume that $\ran(J)=\ran(J_1^b)$ for some $b\in X$ and let $F\in\mathcal{F}_{J,K}$. Then there exists $g\colon U_K\to\ran(F)$ such that for every $j\in Q(J_1^b)$ we have $F(x,y,z)=g\circ S_j(x,y,z)$ for every $(x,y,z)\in D_{J,K}$ such that $(j\circ J(x,y),z)\in D_K$, where $S_j\colon X\times Y\times Z\to U_K$ is defined by $S_j(x,y,z) = K(j\circ J(x,y),z)$.
\end{theorem}

\begin{example}
Let $A$ be a real interval, let $J,K\colon A^2\to A$ be defined by $J(x,y)=K(x,y)=\frac{x+y}{2}$, and let $F\colon A^3\to\R$ be a function for which there exist $G,H\colon A^2\to\R$ such that
$$
F(x,y,z) ~=~ \textstyle{G(J(x,y),z) ~=~ H(x,K(y,z))},\qquad x,y,z\in A.
$$
If $A=\R$, then $\ran(K)=\ran(K_2^0)$ and by Theorem~\ref{thm:main1} the function $F$ is of the form $F(x,y,z)=f(x+y+z)$ for some $f\colon\R\to\R$. If $A=[0,1]$, then we have $\ran(K)\neq\ran(K_2^a)$ for every $a\in [0,1]$. Let us then use Theorem~\ref{thm:main1a} by considering the sets
$$
D_J~=~[0,1]^2\quad\text{and}\quad D_K~=~\{(y,z)\in [0,1]^2\mid y+z\leqslant 1\}.
$$
We then see that $\ran(K|_{D_K})=\ran((K|_{D_K})_2^0)$ and $D_{J,K}=[0,1]\times D_K$. Also, we may define $k\colon [0,\frac{1}{2}]\to [0,1]$ by $k(x)=2x$. By Theorem~\ref{thm:main1a}, for every $F\in\mathcal{F}_{J,K}$ there exists $f\colon [0,2]\to\R$ such that
$$
F(x,y,z)~=~f(x+y+z)\qquad \text{when}~y+z\leqslant 1.
$$
Similarly, considering the set
$$
D'_K ~=~ \{(y,z)\in [0,1]^2\mid y+z\geqslant 1\},
$$
we see that $\ran(K|_{D'_K})=\ran((K|_{D'_K})_2^1)$ and then there exists $f'\colon [1,3]\to\R$ such that
$$
F(x,y,z) ~=~ f'(x+y+z)\qquad \text{when}~y+z\geqslant 1.
$$
It follows that on the whole domain $[0,1]^3$ we have $F(x,y,z)=f(x+y+z)$ for some $f\colon [0,3]\to\R$.\qed
\end{example}

Note that the assumption $(x,k\circ K(y,z))\in D_J$ of Theorem~\ref{thm:main1a} need not be satisfied for every $(x,y,z)\in D_{J,K}$. The following example illustrates this case.

\begin{example}
Let $J$ and $K$ be the real functions defined on
$$
D ~=~ \{(x,y)\in \R^2 \mid y\neq x+1\}
$$
by
$$
J(x,y) ~=~ K(x,y) ~=~ \frac{x+y}{2(x-y+1)}
$$
and let $F\colon C\to\R$ be a function for which there exist $G,H\colon \R^2\to\R$ such that
\begin{equation}\label{eq:rzg5}
F(x,y,z) ~=~ \textstyle{G(J(x,y),z) ~=~ H(x,K(y,z))},\qquad (x,y,z)\in C,
\end{equation}
where $C=\{(x,y,z)\in\R^3\mid y\neq x+1,~z\neq y+1\}$.

Since $K(x,x)=x$, we have $\ran(K)=\R$. We also have $\ran(K_2^a)=\R\setminus\{\frac{1}{2}\}$ if $a\in\R\setminus\{\frac{1}{2}\}$ and $\ran(K_2^{1/2})=\{\frac{1}{2}\}$, which shows that $\ran(K)\neq \ran(K_2^a)$ for every $a\in \R$. We then can apply Theorem~\ref{thm:main1a} if we choose an appropriate restriction of function $K$. Let us use the notation of Problem~\ref{prob:3} and Theorem~\ref{thm:main1a} with the sets
$$
D_J ~=~ D\quad\text{and}\quad D_K ~=~ \textstyle{D\setminus\{(y,z)\in \R^2\mid z=\frac{1}{2}\}}.
$$
We then have $K_2^0(y)=\frac{y}{2(1+y)}$ for every $y\neq -1$ and $\ran(K)=\ran(K_2^0)=\R\setminus \{\frac{1}{2}\}$. The unique quasi-inverse of $K_2^0$ is the function $k\colon \R\setminus\{\frac{1}{2}\} \to \R\setminus\{-1\}$ defined by $k(t)=\frac{2t}{1-2t}$. Moreover we have
$$
D_{J,K} ~=~ \textstyle{\{(x,y,z)\in\R^3\mid y\neq x+1,~ z\neq y+1,~ z\neq \frac{1}{2}\}}
$$
and the condition $(x,k\circ K(y,z))\in D_J$ is then equivalent to $\frac{y+z}{1-2z}\neq x+1$. It follows from Theorem~\ref{thm:main1a} that for every function $F\colon D_{J,K} \to\ran(F)$ in $\mathcal{F}_{J,K}$ there is a function $f\colon \R\to\ran(F)$ such that
\begin{equation}\label{eqn:gh}
F(x,y,z)~=~f\Big(\frac{x+y+z-2xz}{2+2x-2y-6z-4xz}\Big),
\end{equation}
for every $(x,y,z)\in \R^3$ such that $y\neq x+1$, $z\neq y+1$, $z\neq \frac{1}{2}$, and $\frac{y+z}{1-2z}\neq x+1$.

We can apply Theorem~\ref{thm:main2a} similarly. Let us use the corresponding notation with the sets
$$
D_K ~=~ D\quad\text{and}\quad  D_J ~=~ \textstyle{D\setminus\{(x,y)\in\R^2\mid x=-\frac{1}{2}\}}.
$$
We then have $J_1^0(y)=\frac{y}{2(1-y)}$ for every $y\neq 1$ and $\ran(J)=\ran(J_1^0)=\R\setminus\{-\frac{1}{2}\}$. The unique quasi-inverse of $J_1^0$ is the function $j\colon \R\setminus\{-\frac{1}{2}\}\to \R\setminus\{1\}$ defined by $j(t)=\frac{2t}{1+2t}$. Moreover we have
$$
D_{J,K} ~=~ \textstyle{\{(x,y,z)\in\R^3\mid y\neq x+1,~ z\neq y+1,~ x\neq -\frac{1}{2}\}}
$$
and the condition $(j\circ J(x,y),z)\in D_K$ is then equivalent to $\frac{x+y}{1+2x}\neq z-1$. It follows from Theorem~\ref{thm:main2a} that for every function $F\colon D_{J,K} \to\ran(F)$ in $\mathcal{F}_{J,K}$ there is a function $g\colon \R\to\ran(F)$ such that
\begin{equation}\label{eqn:gh02}
F(x,y,z) ~=~ g\Big(\frac{x+y+z+2xz}{2+6x+2y-2z-4xz}\Big),
\end{equation}
for every $(x,y,z)\in \R^3$ such that $y\neq x+1$, $z\neq y+1$, $x\neq -\frac{1}{2}$, and $\frac{x+y}{1+2x}\neq z-1$.

Now, let us consider the function $F$ given in \eqref{eqn:gh} and \eqref{eqn:gh02} on the domain
$$
E ~=~  \textstyle{\{(x,y,z)\in\R^3\mid y\neq x+1,~ z\neq y+1,~ x\neq -\frac{1}{2},~ z\neq \frac{1}{2},~\frac{y+z}{1-2z}\neq x+1,~\frac{x+y}{1+2x}\neq z-1\}}.
$$
Substituting $x=-\frac{1}{3}(t+1)$, $y=t$, and $z=1$ in \eqref{eqn:gh} and \eqref{eqn:gh02}, we obtain
\begin{equation}\label{eqn:mpt01}
f\Big(-\frac{2t+2}{2t+5}\Big) ~=~ g(0), \qquad  t\in\R\setminus\textstyle{\{-\frac{5}{2},0,\frac{1}{2}\}}.
\end{equation}
Similarly, substituting $x=-\frac{1}{5}(t+2)$, $y=t$, and $z=2$, we obtain
\begin{equation}\label{eqn:mpt02}
f\Big(-\frac{4t+8}{2t+19}\Big) ~=~ g(0), \qquad  t\in\R\setminus\textstyle{\{-\frac{19}{2},\frac{1}{2},1\}}.
\end{equation}
It follows from conditions \eqref{eqn:mpt01} and \eqref{eqn:mpt02} that $f\colon \R\to\ran(F)$ is a constant map on $\R\setminus\{-\frac{1}{2}\}$. Using \eqref{eqn:gh} it is then easy to see that $F$ is constant on $E$.

Let us now consider equation \eqref{eq:rzg5} when $x=-\frac{1}{2}$. We have
$$
\textstyle{F(-\frac{1}{2},y,z) ~=~ \textstyle{G(-\frac{1}{2},z) ~=~ H(-\frac{1}{2},K(y,z))}}
$$
for every $(y,z)\in\R^2$ such that $y\neq\frac{1}{2}$ and $z\neq y+1$. It follows that the identity
$$
\textstyle{H(-\frac{1}{2},K(y,z)) ~=~ H(-\frac{1}{2},K(y',z))}
$$
holds for any $y,y'\in\R\setminus\{\frac{1}{2}\}$ and any $z\in\R\setminus\{y+1,y'+1\}$. Since $\ran(K_2^z)=\R\setminus\{\frac{1}{2}\}$ for every $z\neq\frac{1}{2}${\,}, we see that the function $t\mapsto H(-\frac{1}{2},t)$ is constant on $\R\setminus\{\frac{1}{2}\}$. Therefore $F$ is constant on $\{(-\frac{1}{2},y,z)\in\R^3\mid z\neq\frac{1}{2},~y\neq\frac{1}{2},~z\neq y+1\}$. We can show similarly that $F$ is constant on $\{(x,y,\frac{1}{2})\in\R^3\mid x\neq -\frac{1}{2},~y\neq -\frac{1}{2},~y\neq x+1\}$. Finally, $F$ is clearly constant on $\{(-\frac{1}{2},y,\frac{1}{2})\in\R^3\mid y\neq \pm\frac{1}{2}\}$. Now, in order to complete the resolution of this exercise, it would remain to know whether or not these constant values are related and to search for the behavior of $F$ when $\frac{y+z}{1-2z}= x+1$ or $\frac{x+y}{1+2x}= z-1$.\qed
\end{example}

%---------------------------------------------------------------------------------------------- Acknowledgments
\section*{Acknowledgments}

This research is supported by the internal research project R-AGR-0500-MRO3 of the University of Luxembourg.

%---------------------------------------------------------------------------------------------- Bibliography


\begin{thebibliography}{99}

\bibitem{Acz66}
J.~Acz\'el.
\newblock {\em Lectures on functional equations and their applications}.
\newblock Mathematics in Science and Engineering, Vol. 19. Academic Press, New York, 1966.

\bibitem{Mak05}
G.~Maksa.
\newblock Quasisums and generalized associativity.
\newblock {\em Aequationes Math.} 69:6--27, 2005.

\bibitem{Mar10}
J.-L.~Marichal.
\newblock Solving Chisini's functional equation.
\newblock {\em Aequationes Math.} 79:237--260, 2010.

\bibitem{SchSkl83}
B.~Schweizer and A.~Sklar.
\newblock {\em Probabilistic metric spaces}.
\newblock North-Holland Series in Probability and Applied Mathematics.
\newblock North-Holland Publishing Co., New York, 1983.
\newblock (New edition in: Dover Publications, New York, 2005).

\end{thebibliography}
\end{document}